\documentclass{amsart}
\usepackage[utf8]{inputenc}

\usepackage[margin=1in]{geometry}
\usepackage{amsmath,amssymb,amsthm}
\usepackage{enumerate}
\usepackage{mathrsfs}
\usepackage{mathtools}
\counterwithin{equation}{section}
\usepackage{cancel}

\usepackage{tcolorbox}

\usepackage{pdfpages}

\usepackage[foot]{amsaddr}

\theoremstyle{definition}
\newtheorem{definition}{Definition}[section]
\newtheorem{theorem}[definition]{Theorem}

\newtheorem{lemma}[definition]{Lemma}

\newtheorem{notation}[definition]{Notation}

\newtheorem{corollary}[definition]{Corollary}

\title{ \vspace{-.15in}\textbf{On Split Steinberg Modules and Steinberg Modules} }

\author{\vspace{-.1in} Daniel Armeanu}
\address{The authors were supported by nsf grant dms-2202943.}
\author{Jeremy Miller \vspace{-.32in}}

\date{\today}

\usepackage{hyperref}
\hypersetup{
    colorlinks=true, 
    linkcolor=blue,
    filecolor=blue, 
    urlcolor=blue,
    citecolor=blue
}

\begin{document}
\maketitle

\begin{abstract}
    Answering a question of Randal-Williams, we show the natural maps from split Steinberg modules of a Dedekind domain to the associated Steinberg modules are surjective.
\end{abstract}

\section{Introduction}

The goal of this paper is to show that split Steinberg modules of a Dedekind domain surjects onto the corresponding Steinberg modules. This relates an important representation in the theory of homological stability with an important representation in the theory of duality groups. We begin by reviewing definitions.

Fix a Dedekind domain $\Lambda$ and a rank-$n$ projective $\Lambda$-module $M$. The Tits building $T(M)$ is defined to be the geometric realization of the poset of proper summands of $M$ ordered by inclusion. By the Solomon-Tits theorem, there is a homotopy equivalence $T(M) \simeq \vee S^{n-2}$ \cite{Solomon} (see also \cite[Lemma 2.3]{CFP}) and we define the Steinberg module as $$St(M):=\widetilde H_{n-2}(T(M)).$$ Similarly, the split Tits building or Charney building $\widetilde T(M)$ is defined as the geometric realization of the poset of pairs of proper submodules $(P,Q)$ with $P \oplus Q = M$ ordered by inclusion on the first factor and reverse inclusion on the second factor. Charney \cite[Theorem 1.1]{Charney} proved that $\Tilde{T}(M) \simeq \vee S^{n-2}$ and we call its top homology the split Steinberg module or Charney module $$\Tilde{St}(M):=\widetilde H_{n-2}(\widetilde T(M)).$$ There is a natural map $$\Tilde{St}(M) \to St(M) $$ induced by forgetting a complement. Randal-Williams \cite[Theorem 3.3]{ran} showed this map is surjective if $rank(M) \leq 4$ and asked if surjectivity holds for all finitely-generated $M$ \cite[Remark 3.4]{ran}. We answer this question in the affirmative. 

\begin{theorem} \label{main}
    Let $\Lambda$ be a Dedekind domain and let $M$ be a finitely-generated projective $\Lambda$-module. The map $(P,Q) \mapsto Q$ induces a surjective homomorphism $\Tilde{St}(M) \to St(M)$.
\end{theorem}

 We consider the map $(P,Q) \mapsto Q$ instead of $(P,Q) \mapsto P$ because that is what appears in Randal-Williams \cite{ran}, but using that the realization of a poset and its opposite are homeomorphic it follows that the map $(P,Q) \mapsto P$ induces a surjection as well. When $\Lambda$ is Euclidean, work of Ash--Rudolph \cite{AshRudolph} implies that $St(M)$ is generated by integral apartment classes. These classes are in the image of $\Tilde{St}(M)$ so Theorem \hyperlink{main}{1.1} is straightforward for Euclidean domains. However, for many Dedekind domains, the Steinberg modules are not generated by integral apartment classes \cite{CFP,MPWY} so a different argument is required.

Steinberg modules are important objects in representation theory, duality for arithmetic groups \cite{BoSe}, and algebraic $K$-theory \cite{Quillen-Ki}. In contrast, split Steinberg modules are primarly used to study homological stability \cite{Charney,Hepworth-Edge,e2cellsIII,e2cellsIV,KMP,BMS}. 


\subsection*{Acknowledgments}
We thank Oscar Randal-Williams for helpful conversations.

\section{Definitions and known results}

In this section, we collect basic definitions and previously known results related to posets.

\begin{notation}
    Let $M$ be a finitely-generated projective module over a Dedekind domain $\Lambda$. Let $K$, $V$ be proper summands of $M$. Let $T_M$ be the poset of proper nonzero summands of $M$ ordered by inclusion. Let $S_M$ be the poset of pairs $(P,Q)$ of proper summands with $P \oplus Q = M$ with $(P,Q) < (P', Q')$ if $P \subsetneq P'$ and $Q \supsetneq Q'$. Define the following subposets of $S_M$: \begin{itemize}
        \item $S_M(\subseteq , \supseteq V) = \{(P,Q) \in S_M: Q \supseteq V\}$,
        \item $S_M(\subseteq K, \supseteq) = \{(P,Q) \in S_M: P \subseteq K\}$,
        \item $S_M(\subseteq K, \supseteq V) = \{(P,Q) \in S_M: P \subseteq K ,Q \supseteq V\}$.
    \end{itemize}
\end{notation}
Note $|T_M|  \cong T(M)$ and $|S_M| \cong S(M)$.
We will recall upper and lower links in posets.
\begin{notation}
    Let $P$ be a poset and let $a \in P$. Define the following subposets\begin{itemize}
      
        \item $Link_{P}(a) = \{b \in P: b < a \text{ or } b > a\}$,
        \item $Link_{P}^{>}(a) = \{b \in P: b > a \}$,
        \item $Link_{P}^{<}(a) = \{b \in P: b < a \}$.
    \end{itemize}
\end{notation}
Note $|Link_{P}(a)| \cong |Link_{P}^{>}(a)| * |Link_{P}^{<}(a)|$ where $*$ denotes join. The next definition is needed to state Lemma \ref{CP}.  
\begin{definition}
    Let $A$ and $B$ be posets. The height of $a \in A$ denoted $ht(a)$ is the maximal $k$ such that there exists a chain $a_0 \lneq a_1 \lneq ... \lneq a_k = a$ in $A$. Given a poset map $F: A \rightarrow B$, the poset fiber of $b \in B$ denoted $ F_{\leq b}$ is $\{a \in A :$ $F(a) \leq b\}$. 
\end{definition}

The following is Church and Putman \cite[Proposition 2.3]{CP} (see also Quillen \cite[Theorem 9.1]{Quillen-Poset}).

\begin{lemma} [\textbf{Church-Putman}] \label{CP}
    Let $A$ and $B$ be posets. Fix $m \ge 0$ and let $F: A \rightarrow B$ be a map of posets. Assume $B$ is Cohen-Macaulay of dimension $d$ and that for all $b \in B$, the poset fiber $ F_{\leq b}$ is $(ht(b) + m)$-spherical. Then $\widetilde{H}_{i}(A) \rightarrow\widetilde{H}_{i}(B)$ is an isomorphism when $i < d + m$ and surjective for $i \leq d + m$.
\end{lemma}

In particular, the case when $m = 0$ will give surjectivity of the induced map $\Tilde{St}$(M) $\rightarrow St(M)$. \newline

The following is a well-known result from Combinatorial Morse Theory. It is implicit in Charney \cite[Proof of Theorem 1.1]{Charney}; see also Bestvina \cite{Bestvina-MorseTheory}.
\begin{lemma} \label{2.4}
    Let $X$ be a poset. Let $Y \subseteq X$ and $V = X  - Y$. Assume for all distinct $v_1,v_2 \in V$ that $v_1$ and $v_2$ are not comparable. Then
    $$|X| \simeq |Y| \cup \underset{v \in V}{\bigcup} Cone \Big(|Link(v) \cap Y|\Big).$$ In particular, if $|Y| \simeq \vee S^{n}$ and $|Link(v) \cap Y| \simeq \vee S^{n-1}$ for all $v \in V$, then $|X| \simeq \vee S^{n}$.
\end{lemma}
The following is the Solomon-Tits Theorem; see Quillen \cite[Example 8.2]{Quillen-Poset} in the case where $\Lambda$ is a field and see \cite[Remark 2.8]{MPWY} to reduce the Dedekind domain case to the field case. 
\begin{theorem}[\textbf {Solomon-Tits}] \label{sol}
    Let $M$ be a rank-$n$ projective module over a Dedekind domain $\Lambda$. $T_M$ is Cohen-Macaulay of dimension $n-2$.
\end{theorem}

The following due to Charney \cite[Theorem 1.1]{Charney}. 
In particular, we will need $|S_M(\subseteq H, \supseteq)| \simeq \vee S^{n-k-1}$.

\begin{theorem}[\textbf{Charney}] \label{2.5}
    Let $\Lambda$ be a Dedekind domain. Let $M$ be a rank-$n$ projective $\Lambda$-module. Let $H$ and $L$ be rank-$(n-1)$ and rank-$1$ summands of $M$ respectively. Then $|S_M|$, $|S_M(\subseteq, \supseteq L)|$, and $|S_M(\subseteq H, \supseteq)|$ are each homotopy equivalent to a wedge (over possibly different indexing sets) of spheres of dimension $n-2$. 
\end{theorem}

\section{Main Result}

In this section, we prove Theorem \ref{main}, which states that the map $(P,Q) \mapsto Q$ induces a surjective homomorphism $\Tilde{St}(M) \rightarrow St(M)$ for $M$ a finitely-generated projective module over a Dedekind domain. The bulk of this section will be spent proving $S_M(\subseteq,\supseteq V)$ is highly-connected. This will let us apply Lemma \ref{CP} for $m = 0$ to deduce $\Tilde St(M)$ surjects onto $St(M)$.

\begin{theorem} \label{3.1}
    Let $M$ be a rank-$n$ projective module over a Dedekind domain $\Lambda$ and let $V$ be a proper nonzero rank-$k$ summand of $M$. Then $S_M(\subseteq, \supseteq V) \simeq \vee S^{n-k-1}$.
\end{theorem}

\begin{proof}
    The theorem is vacuously true for $n = 0$ and $n = 1$ as $M$ has no proper summands. Theorem \ref{2.5} proven by Charney gives the case when $n=2$ and $n=3$. Randal-Williams has proven the case when $n=4$; see \cite[Theorem 3.3]{ran}. \newline
    
    As noted above, the base case is vacuous. By induction we will assume that the theorem holds for modules $M$ with rank($M$)  $<n$. Now fix $M$ to be a rank-$n$ projective module over a Dedekind domain $\Lambda$ and let $V$ be a proper nonzero rank-$k$ summand of $M$. The result is true when $k = 1$ by Theorem \ref{2.5}, so assume $k \geq 2$.

    The following filtration will be used to show $|S_M(\subseteq, \supseteq{V})|  \simeq \vee S^{n-k-1}$. Choose $L$ a rank-$1$ summand of $V$ and  $H$ a rank-$(n-1)$ summand of $M$ with $H \oplus L = M$ (such an $L$ and $H$ always exist and this implies $V\ \cancel{\subseteq}\ H$). We define the following: \begin{itemize}
        \item $X := S_M(\subseteq, \supseteq{V})$,
        \item $X_0 := S_M(\subseteq{H}, \supseteq{V}) \subsetneq X$,
        \item $T_i := \{(A,B)\in X: rank(A) = n-k-1-i, A\ \cancel{\subseteq}\  H\}$,
        \item $X_i = X_0$ $\cup$ $T_0$ $\cup$ ... $\cup$ $T_{i} \subseteq X.$
        
    \end{itemize}
Note $X = X_{n-k-2}$. \newline

\noindent {\bf Claim 1: $|X_0| \simeq \vee S^{n-k-1}$} \newline 

The maps $(P,Q) \mapsto (P,Q \cap H)$ and $(Z, W) \mapsto (Z, W \oplus L)$ for $(P,Q) \in X_0$ and $(Z,W) \in S_{H}(\subseteq, \supseteq{H \cap V})$ give isomorphisms $$X_0 = \{(P,Q): P \subseteq H, Q \supseteq V\}\ \overset{\cong}{\longleftrightarrow} \{ (Z,W): Z \oplus W = H, W \supseteq H \cap V\} = S_{H}(\subseteq, \supseteq{H \cap V}).$$

\noindent In order for this isomorphism to make sense we must have $P \neq H$, as otherwise $(P,Q \cap H)$ is not an element of $S_{H}(\subseteq, \supseteq H \cap V)$, but this is indeed the case as $k \geq 2$ so $rank(P) \leq n-2  < n-1 = rank(H).$ \newline

Recall $V\ \cancel{\subseteq}\ H$ and $rank(H) < rank(M) = n$, so $rank(H \cap V) = k-1 > 0$, hence $$|X_0| \cong \newline |S_{H}(\subseteq, \supseteq{H \cap V})| \simeq \vee{S^{n-1-(k-1)-1}} = \vee{S^{n-k-1}}$$ by induction. This establishes Claim 1.\newline

\noindent {\bf Claim 2:} Let $(A,B) \in T_{i}$. $|Link_{X}^{>}(A,B) \cap X_{i-1}| \simeq \vee S^{i}$ \newline

We will first show that $Link_{X}^{>}(A,B) \cap X_{i-1} = Link_{X}^{>}(A,B)$. Observe that
\allowdisplaybreaks
\begin{flalign*}
    & \text{$Link_{X}^{>}(A,B) \cap X_{i-1} =$} & \\
    & \text{$\{(P,Q) \in X: A \subsetneq P, Q \subsetneq B, \text{ and } (P,Q) \in X_{i-1}\} = $} & \\
    & \text{$\{(P,Q)\in X: A \subsetneq P, Q \subsetneq B, \text{ and either } (P,Q) \in X_{0} \text{ or } (P,Q) \in T_j \text{ for } j < i\} = $} & \\
    & \text{$\{(P,Q)\in X: A \subsetneq P, Q \subsetneq B, \text{ and either } P \subseteq H \text{ or } rank(P) = n-j-k-1 \text{ and } P\ \cancel{\subseteq}\ H\}.$} & \\
    & \text{$\{(P,Q)\in X: A \subsetneq P, Q \subsetneq B, \text{ and either } P \subseteq H \text{ or } rank(Q) = j + k + 1 < i + k +1= rank(B) \text{ and } P\ \cancel{\subseteq}\ H\}.$} & \\
    & \text{The condition $rank(Q) < rank(B)$ follows from the fact that $Q \subsetneq B$. Thus } & \\
    & \text{$Link_{X}^{>}(A,B) \cap X_{i-1} = \{(P,Q) \in X: A \subsetneq P, Q \subsetneq B, \text{ and either } P \subseteq H \text{ or } P\ \cancel{\subseteq}\ H\} = $} & \\
    & \text{$\{(P,Q) \in X: A \subsetneq P, Q \subsetneq B\} = Link_X^>(A,B).$} &
\end{flalign*}
The maps $(P,Q) \mapsto (P\cap B, Q)$ and $(Z,W) \mapsto (Z \oplus A, W)$ for $(P,Q) \in Link_{X}^{>}(A,B)$ and $(Z,W) \in \newline S_B(\subseteq, \supseteq V)$ give isomorphisms between $$Link_{X}^{>}(A,B)\overset{\cong}{\longleftrightarrow} S_B(\subseteq, \supseteq V).$$ Since $rank(B) = k+i+1$ and $rank(V) = k$, $$|S_B(\subseteq, \supseteq V)| \simeq \vee S^{(k + i + 1) - k - 1} = \vee S^{i}$$ by the induction hypothesis. This establishes the claim. \newline

\noindent {\bf Claim 3:} Let $(A,B) \in T_{i}$. $|Link_{X}^{<}(A,B) \cap X_{i-1}| \simeq \vee S^{n-k-i-3}$ \newline

By a similar argument, 
\begin{flalign*}
    & \text{$Link_{X}^{<}(A,B) \cap X_{i-1} =$} & \\
    & \text{$\{(P,Q) \in X: P \subsetneq A, Q \supsetneq B, \text{ and } (P,Q) \in X_{i-1}\} = $} & \\
    & \text{$\{(P,Q) \in X: P \subsetneq A, Q \supsetneq B, \text{ and either } (P,Q) \in X_{0}\ \text{ or } (P,Q) \in T_j \text{ for } j < i\} = $} & \\
    & \text{$\{(P,Q) \in X: P \subsetneq A, Q \supsetneq B, \text{ and either } P \subseteq H\ \text{ or } P\ \cancel{\subseteq}\ H \text{ and } rank(P) = n -j - k - 1\}=$} & \\
    & \text{$\{(P,Q) \in X: P \subsetneq A, Q \supsetneq B, \text{ and either } P \subseteq H\ \text{ or } P\ \cancel{\subseteq}\ H \text{ and } rank(Q) = j + k + 1< i + k + 1 =  rank(B)\}.$} & \\
    & \text{Notice that since $Q \supsetneq B$ it follows that $rank(Q) > rank(B)$ so the second case cannot happen. Thus, } & \\
    & \text{$Link_{X}^{<}(A,B) \cap X_{i-1} =$} & \\
    & \text{$\{(P,Q) \in X: P \subsetneq A, Q \supsetneq B, \text{ and } P \subseteq H\} = $} & \\
    & \text{$\{(P,Q) \in X: P \subseteq A \cap H, Q \supsetneq B\}.$} &
\end{flalign*}
The maps $(P,Q) \mapsto (P, Q \cap A) $ and $(Z,W) \mapsto (Z, W \oplus B)$ for $(P,Q) \in Link_{X}^{>}(A,B)$ and $(Z,W) \in \newline S_B(\subseteq, \supseteq V)$ give isomorphisms between $$\{(P,Q) \in X: P \subsetneq A \cap H, Q \supseteq B\} \overset{\cong}{\longleftrightarrow} S_{A}(\subseteq H \cap A, \supseteq).$$ 
Observe $corank(A \cap H) = 1$ inside $A$ and $rank(A) = n-k-i-1$ hence by Theorem \ref{2.5} $$|S_{A}(\subseteq H \cap A, \supseteq)| \simeq \vee S^{n-k-i-3}.$$ This establishes the claim. \newline

\noindent {\bf Claim 4:} $|X_j| \simeq \vee S^{n-k-1}$ for all $0 \leq j \leq n-k-2$ \newline

We will show this by induction on $j$. The base case $j=0$ is given by Claim 1. To show $|X_j| \simeq \vee S^{n-k-1}$, we verify the hypothesis of Lemma \ref{2.4} with $X = X_j, Y = X_{j-1},\text{and }V = T_j$. Distinct elements of $T_j$ are not comparable and $|X_{j-1}| \simeq \vee S^{n-k-1}$ by the induction hypothesis. Lastly, using Claim 2 and Claim 3 $$|Link_X(A,B) \cap X_{j-1}| \cong |Link_{X}^{>}(A,B) \cap X_{j-1}| * |Link_{X}^{<}(A,B) \cap X_{j-1}| \simeq \vee S^{j} * \vee S^{n-k-j-3} \cong \vee S^{n-k-2}.$$ Hence by Lemma \ref{2.4}, $|X_j| \simeq \vee S^{n-k-1}.$ This establishes the claim. \newline

The theorem follows since $X = X_{n-k-2}.$
\end{proof}

The following corollary of Theorem \ref{3.1} is not needed to prove Theorem \ref{main}, so we do not provide a proof, but it can be obtained by dualizing and following Charney's strategy in Step 2 of \cite[Theorem 1.1]{Charney}. 
\begin{corollary} \label{cor}
    Let $M$ be a rank-$n$ projective module over a Dedekind domain $\Lambda$. Let $K_0$ be a corank-$k$ summand of M. Then $S_M(\subseteq K_0, \supseteq) \simeq \vee S^{n-k-1}$. 
\end{corollary}

 We now prove Theorem \ref{main} which states the map from $S_M$ to $T_M^{op}$ given by $(P,Q) \mapsto Q$ induces a surjective map $\Tilde{St}(M) \to St(M)$.

\begin{proof}
    Recall $|S_M| = \Tilde{T}(M)$ and $|T_M| = |T_M^{op}| = T(M)$.  We verify the hypotheses of Lemma \ref{CP} with $A = S_M,B = T_M^{op}$, $m = 0$, and $(P,Q) \mapsto Q$ the poset map.  $T_M^{op}$ is Cohen-Macaulay by the Solomon-Tits Theorem here Theorem \ref{sol}. It remains to check the poset fiber of $V \in T_M^{op}$ is $ht(V)$-spherical. Let $rank(V) = k < n$. A maximal chain containing $V$ in $T_M^{op}$ is of the form $A_{n-1}  \supsetneq  A_{n-2} \supsetneq ... \supsetneq V$, where the $A_i$ are rank-$i$ summands of $M$. This gives $ht(V) = n - k -1$. Using Theorem \ref{3.1} we observe $$F_{\leq V} = \{(P,Q) \in S_M: Q \supseteq{V}\} = S_M(\subseteq , \supseteq{V}) \simeq \vee S^{n-k-1}.$$ Hence, the induced map from $$\Tilde{St}(M) = \widetilde{H}_{n-2}(S) \to \widetilde{H}_{n-2}(T) = St(M)$$ is surjective. 
\end{proof}

\bibliographystyle{amsalpha}
\bibliography{refs}

\vspace{8pt}
\footnotesize{\textit{Email address:} darmeanu@purdue.edu} \newline

\footnotesize{Purdue University, Department of Mathematics, 150 N. University, West Lafayette IN, 47907, USA} \newline 

\footnotesize{\textit{Email address:} jeremykmiller@purdue.edu} \newline

\footnotesize{Purdue University, Department of Mathematics, 150 N. University, West Lafayette IN, 47907, USA}

\end{document}